\documentclass[12pt]{amsart}
\usepackage{amssymb}
\usepackage{hyperref}

\newtheorem{theorem}{Theorem}[section]
\newtheorem{definition}[theorem]{Definition}
\newtheorem{proposition}[theorem]{Proposition}
\newtheorem{corollary}[theorem]{Corollary}
\newtheorem{lemma}[theorem]{Lemma}

\begin{document}

\author{T. Banica}
\address{T.B.: Department of Mathematics, Cergy-Pontoise University, 95000 Cergy-Pontoise, France. {\tt teodor.banica@u-cergy.fr}}

\author{J. Bhowmick}
\address{J.B.: Faculty of Mathematics and Natural Sciences, University of Oslo, Po. Box 1032 Blindern, N-0315 Oslo, Norway. {\tt jyotishb@math.uio.no}}

\author{K. de Commer}
\address{K.D.: Department of Mathematics, Cergy-Pontoise University, 95000 Cergy-Pontoise, France. {\tt kenny.de-commer@u-cergy.fr}}

\title{Quantum isometries and group dual subgroups}


\begin{abstract}
We study the discrete groups $\Lambda$ whose duals embed into a given compact quantum group, $\widehat{\Lambda}\subset G$. In the matrix case $G\subset U_n^+$ the embedding condition is equivalent to having a quotient map $\Gamma_U\to\Lambda$, where $F=\{\Gamma_U|U\in U_n\}$ is a certain family of groups associated to $G$. We develop here a number of techniques for computing $F$, partly inspired from Bichon's classification of group dual subgroups $\widehat{\Lambda}\subset S_n^+$. These results are motivated by Goswami's notion of quantum isometry group, because a compact connected Riemannian manifold cannot have non-abelian group dual isometries.
\end{abstract}

\maketitle

\section*{Introduction}

The quantum groups were introduced in the mid-eighties by Drinfeld \cite{dri} and Jimbo \cite{jim}. Soon after, Woronowicz developed a powerful axiomatization in the compact case \cite{wo1}, \cite{wo2}, \cite{wo3}. His axioms use $\mathbb C$ as a ground field, and the resulting compact quantum groups have a Haar measure, and are semisimple. The Drinfeld-Jimbo quantum groups $G^q$ with $q\in\mathbb R$ (or rather their compact forms) are covered by this formalism. In particular the examples include the twists $G^{-1}$, whose square of the antipode is the identity.

Building on Woronowicz's work, Wang constructed in the mid-nineties a number of new, interesting examples: the free quantum groups \cite{wa1}, \cite{wa2}. Of particular interest are the quantum automorphism groups $G^+(X)$ of the finite noncommutative spaces $X$, constructed in \cite{wa2}. For instance in the simplest case $X=\{1,\ldots,n\}$, the quantum group $G^+(X)$, also known as ``quantum permutation group'', is infinite as soon as $n\geq 4$. This phenomenon, discovered by Wang in \cite{wa2}, was further investigated in \cite{ban}, the main result there being that we have a fusion semiring equivalence $G^+(X)\sim SO_3$, for any finite noncommutative space $X$ subject to the Jones index type condition $|X|\geq 4$.

The next step was to restrict attention to the classical case $X=\{1,\ldots,n\}$, but to add some extra structure: either a metric, or, equivalently, a colored graph structure. The algebraic theory here was developed in \cite{bbi} and in subsequent papers. Also, much work has been done in connecting the representation theory of $G^+(X)$ with the planar algebra formalism of Jones \cite{jon} and with Voiculescu's free probability theory \cite{vdn}, the connection coming via the Collins-\'Sniady integration formula \cite{csn}, and via Speicher's notion of free cumulant \cite{spe}. For some recent developments in this direction, see \cite{bs1}.

A new direction of research, recently opened up by Goswami \cite{go1}, is that of looking at the quantum isometry groups of noncommutative Riemannian manifolds. As for the finite noncommutative spaces, or the various noncommutative spaces in general, these manifolds have in general no points, but they can be described by their spectral data. The relevant axioms here, leading to the unifying notion of ``spectral triple'', were found by Connes \cite{con}, who was heavily inspired by fractal spaces, foliations, and a number of key examples coming from particle physics and from number theory. See \cite{con}, \cite{cma}.

Following Goswami's pioneering paper \cite{go1}, the fundamentals of the theory were developed in \cite{bg1}, \cite{bg2}, \cite{bg3}. The finite-dimensional spectral triples, making the link with the previous work on finite graphs, were studied in \cite{bgs}. A lot of recent work has been done towards the understanding of the discrete group dual case \cite{bsk}, where several diagrammatic and probabilistic tools are known to apply, cf. \cite{bs1}, \cite{lso}. The other important direction of research is that of investigating the quantum symmetries of Connes' Standard Model algebra, with the work here started in the recent papers \cite{bdd}, \cite{bd+}.

There are, however, two main theoretical questions that still lie unsolved, at the foundations of the theory. One of them is whether an arbitrary compact metric space (without Riemannian manifold structure) has a quantum isometry group or not. For the difficulties in dealing with this question, and for some partial results, we refer to \cite{bs2}, \cite{go2}, \cite{hua}, \cite{qsa}.

The second question, which is equally very important, and which actually belongs to a much more concrete circle of ideas, as we will see in this paper, is as follows:

\medskip
\noindent {\bf Conjecture.} {\em A non-classical compact quantum group $G$ cannot act faithfully and isometrically on a compact connected Riemannian manifold $M$.}
\medskip

The first verification here, going back to \cite{bg1}, shows that the sphere $S^n$ has indeed no genuine quantum isometries. Note that this is no longer true for the various noncommutative versions of $S^n$, such as the Podl\'es spheres \cite{bg3}, or the free spheres \cite{bgo}. 

Another key computation is the one in \cite{bho}, where it was proved, via some quite complex algebraic manipulations, that the torus $\mathbb T^k$ has no genuine quantum isometries either. 

Recently Goswami has shown that a large class of homogeneous spaces have no genuine quantum isometries \cite{go3}. This gives strong evidence for the above conjecture.

The starting point for the present work was the following elementary observation, inspired from the work of Boca on ergodic actions in \cite{boc}:

\medskip
\noindent {\bf Fact.} {\em A non-classical discrete group dual $\widehat{\Lambda}$ cannot act faithfully and isometrically on a compact connected Riemannian manifold $M$.}
\medskip

The proof of this fact is quite standard, using a brief computation with group elements. We should mention that this computation needs a bit of geometry, namely the domain property ($f,g\neq 0\implies fg\neq 0$) of eigenfunctions on connected manifolds.

As a first consequence, all the group dual subgroups $\widehat{\Lambda}\subset G$ of a given quantum isometry group $G=G^+(M)$ must be classical. Thus, we are led to the following notions:

\medskip
\noindent {\bf Definition.} {\em A compact quantum group $G$ is called ``basic'' if it is either classical, or has a non-classical group dual subgroup $\widehat{\Lambda}\subset G$, and ``strange'' otherwise.}
\medskip

We should mention that the terminology here, while being quite natural in view of the above fact, can of course seem a bit ackward. Here are a few more explanations:

-- We assume throughout this paper that our compact quantum groups are of Kac type, so that the various $q$-deformations etc. are not concerned by the above dichotomy.

-- As we will see, most ``basic'' examples of compact quantum groups, such as the compact groups, group duals, easy quantum groups etc. are basic in the above sense.

-- One problem however comes from the Kac-Paljutkin quantum group \cite{kpa}, which, while being a very old and fundamental one, is ``strange'' in the above sense. 

-- Summarizing, the terminology in the above definition is just the best one that we could find, based on what we know, and should be of course taken with caution.

Back to our quantum isometry considerations now, the above fact shows that any compact quantum group contradicting the above conjecture must be strange. So, we are naturally led to the problem of understanding the structure of strange quantum groups.

The problem of deciding whether a given compact quantum group is basic or strange is not trivial, and basically requires solving the following question:

\medskip
\noindent {\bf Question.} {\em What are the discrete groups $\Lambda$ whose duals embed into a given compact quantum group, $\widehat{\Lambda}\subset G$?}
\medskip

It is this latter question, which is purely quantum group-theoretical, that we will investigate here. As a first remark, the techniques for dealing with it don't lack:
\begin{enumerate}
\item The guiding result, obtained by Bichon in \cite{bi2}, is that the group dual subgroups $\widehat{\Lambda}\subset S_n^+$ appear from quotients $\mathbb Z_{n_1}*\ldots*\mathbb Z_{n_s}\to\Lambda$, with $n=\sum n_i$.

\item Yet another key result, obtained in \cite{bve}, is that the ``diagonally embedded'' group dual subgroups $\widehat{\Lambda}\subset O_n^*$ appear from quotients $\mathbb Z^{n-1}\rtimes\mathbb Z\to\Lambda$.

\item Franz and Skalski classified in \cite{fsk} all closed subgroups, so in particular all group dual subgroups, of Sekine's quantum groups \cite{sek}.

\item There are some other quantum groups, all whose group dual subgroups can be computed: for instance the Hajac-Masuda quantum double torus \cite{hma}.

\item The homogeneous spaces of type $\widehat{\Lambda}/(\widehat{\Lambda}\cap U_k^+)$, with $\widehat{\Lambda}\subset U_n^+$ closed subgroup and with $k\leq n$, were investigated in the recent paper \cite{bss}.

\item Finally, the projective representations of compact quantum groups, partly generalizing Bichon's formalism in \cite{bi2}, were investigated in \cite{dec}.
\end{enumerate}

The problem is somehow to put all these ingredients together, by using a convenient formalism. In the matrix case $G\subset U_n^+$ this can be done by using the notion of ``diagonal subgroup'' from \cite{bve}, and we have the following answer to the above question:

\medskip
\noindent {\bf Answer.} {\em The closed subgroups $\widehat{\Lambda}\subset G$ appear from quotients $\Gamma_U\to\Lambda$ of a certain family of discrete groups $F=\{\Gamma_U|U\in U_n\}$ associated to $G$.}
\medskip

With this observation in hand, the above considerations can be organized and further processed. First, we will show that the quantum groups in (2) and their generalizations are basic, and that the quantum groups in (3) and (4) are strange. Also, by using some ideas from (5) and (6), we will recover Bichon's groups in (1). The problem of finding a wide generalization of (1-6) above, in terms of the family $F$, remains however open.

The paper is organized as follows: 1 is a preliminary section, in 2 we present some basic results on quantum isometries, in 3 we discuss the actions of group duals, in 4 we develop the general theory of diagonal subgroups, and in 5 we some present explicit computations, in a number of special cases. The final section, 6, contains a few concluding remarks.

\subsection*{Acknowledgements}

Part of this work was done during visits of J.B. at the IHES in August 2011 and at the Cergy-Pontoise University in November 2011, and during a visit of all three authors at the Banach Center in Warsaw in September 2011. The work of T.B. was supported by the ANR grant ``Granma''.

\section{Quantum isometries}

Let $M$ be a compact Riemannian manifold. That is, $M$ is a smooth real manifold, that we will always assume to be compact, and given with a real, positive definite scalar product $<,>$ on each tangent space $T_xM$, depending smoothly on $x$.

\begin{definition}
Associated to a compact Riemannian manifold $M$ are:
\begin{enumerate}
\item $Diff(M)$: the group of diffeomorphisms $\varphi:M\to M$.

\item $G(M)\subset Diff(M)$: the subgroup of isometries $\varphi:M\to M$.
\end{enumerate}
\end{definition}

We use here the non-standard notation $G(M)$ instead of the usual one $ISO(M)$, because this group will be subject to a ``liberation'' operation $G\to G^+$, and the notation $ISO^+(M)$ is traditionally reserved for the group of orientation-preserving isometries.

Let $\Omega^1(M)$ be the space of smooth 1-forms on $M$, with scalar product:
$$<\omega,\eta>=\int_M<\omega(x),\eta(x)>dx$$

Consider the differential $d:C^\infty(M)\to\Omega^1(M)$, and define the Hodge Laplacian $L:L^2(M)\to L^2(M)$ by $L=d^*d$. Note that we use here the non-standard notation $L$ instead of the usual one $\Delta$, because we prefer to keep $\Delta$ for the comultiplication of the above-mentioned ``liberated'' quantum group $G^+(M)$, to be introduced later on.

Following Goswami \cite{go1}, we will make use of the following basic fact:

\begin{proposition}
$G(M)$ is the group of diffeomorphisms $\varphi:M\to M$ whose induced action on $C^\infty(M)$ commutes with the Hodge Laplacian $L=d^*d$. 
\end{proposition}

In order to present now some quantum group analogues of the above statements, we use the general formalism developed by Woronowicz in \cite{wo1}, \cite{wo2}, \cite{wo3}. Thus, a compact quantum group will be an abstract object $G$, having no points in general, but which is described by a well-defined Hopf $C^*$-algebra ``of functions'' on it, $A=C(G)$.

The axioms for Hopf $C^*$-algebras, found in \cite{wo3}, are as follows:

\begin{definition}
A Hopf $C^*$-algebra is a unital $C^*$-algebra $A$, given with a morphism of $C^*$-algebras $\Delta:A\to A\otimes A$, subject to the following  conditions:
\begin{enumerate}
\item Coassociativity: $(\Delta\otimes id)\Delta=(id\otimes\Delta)\Delta$.

\item $\overline{span}\,\Delta(A)(A\otimes 1)=\overline{span}\,\Delta(A)(1\otimes A)=A\otimes A$.
\end{enumerate}
\end{definition}

The basic example is $A=C(G)$, where $G$ is a compact group, with $\Delta f(g,h)=f(gh)$. The fact that $\Delta$ is coassociative corresponds to $(gh)k=g(hk)$, and the conditions in (2) correspond to the cancellation rules $gh=gk\implies h=k$ and $gh=kh\implies g=k$. 

Conversely, any commutative Hopf $C^*$-algebra is of the form $C(G)$. Indeed, by the Gelfand theorem we have $A=C(G)$, with $G$ compact space, and (1,2) above tell us that $G$ is a semigroup with cancellation. By a well-known result, $G$ follows to be a group.

The other main example is $A=C^*(\Gamma)$, where $\Gamma$ is a discrete group, with comultiplication $\Delta(g)=g\otimes g$. One can prove that any Hopf $C^*$-algebra which is cocommutative, in the sense that $\Sigma\Delta=\Delta$, where $\Sigma(a\otimes b)=b\otimes a$ is the flip, is of this form.

These basic facts, together with some other general results in \cite{wo3}, lead to:

\begin{definition}
Associated to any Hopf $C^*$-algebra $A$ are a compact quantum group $G$ and a discrete quantum group $\Gamma=\widehat{G}$, according to the formula $A=C(G)=C^*(\Gamma)$.
\end{definition}

The meaning of this definition is of course quite formal. The idea is that, with a suitable definition for morphisms, the Hopf $C^*$-algebras form a category $X$. One can define then the categories of compact and discrete quantum groups to be $\widehat{X}$, and $X$ itself, and these categories extend those of the usual compact and discrete groups. See \cite{wo3}.

Let us go back now to the Riemannian manifolds, and to the groups constructed in Definition 1.1. We cannot define their quantum analogues as being formed of ``quantum bijections'' $\varphi:M\to M$, simply because these quantum bijections do not exist: remember, the quantum groups are abstract objects, having in general no points.

So, we will need the ``spectral'' point of view brought by Proposition 1.2. More precisely, following Goswami \cite{go1}, we can formulate the following definition:

\begin{definition}
Associated to a compact Riemannian manifold $M$ are:
\begin{enumerate}
\item $Diff^+(M)$: the category of compact quantum groups acting on $M$.

\item $G^+(M)\in Diff^+(M)$: the universal object with a coaction commuting with $L$.
\end{enumerate}
\end{definition}

In this definition the quantum group actions are defined in terms of the associated coactions, $\alpha:C(M)\to C(M)\otimes C(G)$, which have to satisfy the smoothness assumption $\alpha(C^\infty(M))\subset C^\infty(M)\otimes C(G)$. As for the commutation condition with $L$, this regards the canonical extension of the action to the space $L^2(M)$. See Goswami \cite{go1}.

Observe that we have an inclusion of compact quantum groups $G(M)\subset G^+(M)$, coming from Proposition 1.2. In the disconnected case, this inclusion is in general proper. In the connected case, this inclusion is conjectured to be an isomorphism \cite{go3}.

\section{Some basic results}

Let us first discuss some examples of genuine quantum group actions, in the disconnected case. We use the following notion, due to Wang \cite{wa1}:

\begin{definition}
Given two compact quantum groups $G,H$ we let
$$C(G\;\hat{*}\;H)=C(G)*C(H)$$
with the Hopf algebra operations extending those of $C(G),C(H)$.
\end{definition}

Here the notation $\;\hat{*}\;$ comes from the fact that this operation is dual to the free product operation $*$ for discrete quantum groups, given by $C^*(\Gamma*\Lambda)=C^*(\Gamma)*C^*(\Lambda)$.

We have the following basic examples of isometric quantum group actions:

\begin{proposition}
Let $M=N_1\sqcup\ldots\sqcup N_k$ be a disconnected manifold.
\begin{enumerate}
\item We have an inclusion $G^+(N_1)\;\hat{*}\;\ldots\;\hat{*}\;G^+(N_k)\subset G^+(M)$.

\item If $G(N_i)\neq\{1\}$ for at least two indices $i$, then $G^+(M)\neq G(M)$.
\end{enumerate}
\end{proposition}

\begin{proof}
We use the canonical identification $C(M)=C(N_1)\oplus\ldots\oplus C(N_k)$.

(1) For $i=1,\ldots,k$ let $\alpha_i:C(N_i)\to C(N_i)\otimes C(G_i)$ be isometric coactions of Hopf $C^*$-algebras $C(G_i)$ on the manifolds $N_i$, and let $G=\;\hat{*}\;G_i$. By using the canonical embeddings $C(N_i)\subset C(M)$ and $C(G_i)\subset C(G)$ we can define a map $\alpha:C(M)\to C(M)\otimes C(G)$ by $\alpha(f_1,\ldots,f_k)=\alpha_1(f_1)\ldots\alpha_k(f_k)$, and it follows from definitions that this map is an isometric coaction. With $G_i=G^+(N_i)$, this observation gives the result.

(2) Since we have inclusions $G(N_i)\subset G^+(N_i)$ for any $i$, by taking a dual free product we obtain an inclusion $\hat{*}\;G(N_i)\subset\hat{*}\;G^+(N_i)$. By combining with (1) we obtain an inclusion $\hat{*}\;G(N_i)\subset G^+(M)$, i.e. a surjective morphism $C(G^+(M))\to *C(G(N_i))$. Now since $A,B\neq\mathbb C$ implies that $A*B$ is not commutative, this gives the result.
\end{proof}

Let us investigate now the behavior of $G^+(.)$ with respect to product operations. Given two Riemannian manifolds $M,N$ we can consider their Cartesian product $M\times N$, with scalar product on each tangent space $T_{(x,y)}(M\times N)=T_xM\oplus T_yN$ given by:
$$<u\oplus u',v\oplus v'>=<u,v><u',v'>$$

We use the standard identification $C(M\times N)=C(M)\otimes C(N)$.

\begin{lemma}
$L_{M\times N}=L_M\otimes 1+1\otimes L_N$.
\end{lemma}

\begin{proof}
This follows from the fact that the whole de Rham complex for $M\times N$ decomposes as a ``tensor product'' of the de Rham complexes for $M,N$. First, we have:
$$\Omega^k(M\times N)=\bigoplus_{i+j=k}\Omega^i(M)\otimes\Omega^j(N)$$

Also, the differential is $d=d_M\otimes id+id\otimes d_N$. Thus, we get:
\begin{eqnarray*}
<d^*d(f\otimes g),h\otimes k>
&=&<d_Mf\otimes g+f\otimes d_Ng,d_Mh\otimes k+h\otimes d_Nk>\\
&=&<d_Mf\otimes g,d_Mh\otimes k>+<f\otimes d_Ng,h\otimes d_Nk>\\
&=&<d_M^*d_Mf\otimes g,h\otimes k>+<f\otimes d_N^*d_Ng,h\otimes k>
\end{eqnarray*}

This gives $d^*d=d_M^*d_M\otimes 1+1\otimes d_N^*d_N$, as claimed.
\end{proof}

Observe that the above operation is ``compatible'' with the product operation for graphs in \cite{bbi}, given at the level of adjacency matrices by $d_{X\times Y}=d_X\otimes 1+1\otimes d_Y$.

\begin{theorem}
Assume that $M,N$ are connected and that their spectra $\{\lambda_i\}$ and $\{\mu_j\}$ ``don't mix'', in the sense that we have $\{\lambda_i-\lambda_j\}\cap\{\mu_i-\mu_j\}=\{0\}$. Then:
\begin{enumerate}
\item $G(M\times N)=G(M)\times G(N)$.

\item $G^+(M\times N)=G^+(M)\times G^+(N)$.
\end{enumerate}
\end{theorem}

\begin{proof}
We follow the proof in \cite{bbi}, where a similar result was proved for finite graphs. Since the classical symmetry group is the classical version of the quantum isometry group, it is enough to prove the second assertion, for the quantum isometry groups.

Let $L_M=\sum_\lambda\lambda\cdot P_\lambda$ and $L_N=\sum_\mu\mu\cdot Q_\mu$ be the formal spectral decompositions of $L_M,L_N$. Since we have $L_{M\times N}=L_M\otimes 1+1\otimes L_N$, we get:
$$L_{M\times N}=\sum_{\lambda\mu}(\lambda+\mu)\cdot(P_\lambda\otimes Q_\nu)$$

The non-mixing assumption in the statement tells us that the scalars $\lambda+\mu$ appearing in this formula are distinct. Since the projections $P_\lambda\otimes Q_\nu$  form a partition of the unity, it follows that the above formula is the formal spectral decomposition of $L_{M\times N}$.

We can conclude now as in \cite{bbi}. The universal coaction of $G^+(M\times N)$ must commute with any spectral projection $P_\lambda\otimes Q_\mu$, and hence with both the following projections:
\begin{eqnarray*}
P_0\otimes 1&=&\sum_\mu P_0\otimes Q_\mu\\
1\otimes Q_0&=&\sum_\lambda P_\lambda\otimes Q_0
\end{eqnarray*}

Now since $M,N$ are connected, the above projections are both 1-dimensional. It follows that the universal coaction of $G^+(M\times N)$ is the tensor product of  its restrictions to the images of $P_0\otimes 1$, i.e. to $1\otimes C(N)$, and of $1\otimes Q_0$, i.e. to $C(M)\otimes 1$, and we are done.
\end{proof}

\begin{corollary}
If $M,N$ are connected, without quantum isometries, and their spectra don't mix, then $M\times N$ doesn't have quantum isometries either.
\end{corollary}

\begin{proof}
This is clear from Theorem 2.4.
\end{proof}

Observe that this kind of statement, and the above algebraic technology in general, is far below from what would be needed for attacking the conjecture in the introduction. For instance our results don't cover the torus $\mathbb T^k$, obtained as a ``mixing'' product of $k$ circles, and which is known from \cite{bho} not to have genuine quantum isometries.

\section{Group dual actions}

As already mentioned in the introduction, the results in \cite{bho}, \cite{bg1}, along with the recent ones in \cite{go3}, and also with those in the previous section, give some substantial evidence for the conjectural statement ``$M$ classical and connected implies $G(M)=G^+(M)$''.

One way of attacking this conjecture would be by trying to extend first Goswami's recent results of homogeneous spaces in \cite{go3}. Another possible way would be by trying to extend first Theorem 2.4 above, as to cover the computations in \cite{bho} for the torus.

Yet another method, that we believe to be important as well, is by using group dual subgroups. It is known indeed since the work of Boca \cite{boc} that a compact space cannot have a genuine ergodic group dual action. We will use here the same kind of idea.

We recall that for a connected Riemannian manifold $M$, the eigenfunctions of the Laplacian have the domain property, namely $f,g\neq 0$ implies $fg\neq 0$. This is for instance because the set of zeros of each nonzero eigenfunction of the Laplacian is known to have Hausdorff dimension $\dim M-1$, and hence measure zero. See e.g. \cite{yau}.

By using now the same computation as in \cite{boc}, we get:

\begin{proposition}
A compact connected Riemannian manifold $M$ cannot have genuine group dual isometries.
\end{proposition}

\begin{proof}
Assume that we have a group dual coaction $\alpha:C(M)\to C(M)\otimes C^*(\Gamma)$.

Let $E=E_1\oplus E_2$ be the direct sum of two eigenspaces of $L$. Pick a basis $\{x_i\}$ such that the corepresentation on $E$ becomes diagonal, i.e. $\alpha(x_i)=x_i\otimes g_i$ with $g_i\in\Gamma$. The formula $\alpha(x_ix_j)=\alpha(x_jx_i)$ reads $x_ix_j\otimes g_ig_j=x_ix_j\otimes g_jg_i$, and by using the domain property we obtain $g_ig_j=g_jg_i$. Also, the formula $\alpha(x_i\bar{x}_j)=\alpha(\bar{x}_jx_i)$ reads $x_i\bar{x}_j\otimes g_ig_j^{-1}=x_i\bar{x}_j\otimes g_j^{-1}g_i$, and by using the domain property again, we obtain $g_ig_j^{-1}=g_j^{-1}g_i$. Thus the elements $\{g_i,g_i^{-1}\}$ mutually commute, and with $E$ varying, this shows that $\Gamma$ is abelian. 
\end{proof}

The above result is quite interesting, because it shows that all the group dual subgroups of a given quantum isometry group must be classical. More precisely, let us first divide the compact quantum groups into two classes, as suggested in the introduction:

\begin{definition}
We call a compact quantum group $G$:
\begin{enumerate}
\item ``Strange'', if it is non-classical, and all its group dual subgroups are classical.

\item ``Basic'', if not (i.e. is either classical, or has a non-classical group dual subgroup).
\end{enumerate}
\end{definition}

Observe that this definition is purely algebraic, making no reference to manifolds and to their quantum isometry groups. As we will soon see, most known examples of quantum groups are basic, and this can be usually checked by purely algebraic computations. However, we will see as well that several classes of strange quantum groups exist.

The relation with the quantum isometry groups comes from:

\begin{proposition}
A non-classical compact quantum group $G$ acting isometrically on a compact connected Riemannian manifold $M$ must be strange.
\end{proposition}

\begin{proof}
This is just a reformulation of Proposition 3.1, by using Definition 3.2.
\end{proof}

Thus, a counterexample to the conjecture in the introduction could only come from a strange quantum group. So, let us try to understand what these quantum groups are. 

We begin with a few results on the basic quantum groups. We recall that given two compact quantum groups $G,H$ we can form their product $G\times H$, and their dual free product $G\;\hat{*}\;H$. In the case where $G,H\subset K$ are subgroups of  a given compact quantum group, we can form the generated group $<G,H>\subset K$. Finally, if $H$ is a quantum permutation group, we can form the free wreath product $G\wr_*H$. See \cite{bi1}, \cite{wa1}.

\begin{proposition}
The class of basic quantum groups has the following properties:
\begin{enumerate}
\item It contains all classical groups, and all group duals.

\item It is stable by products, and by taking generating groups.

\item It is stable by dual free products, and free wreath products.
\end{enumerate}
\end{proposition}

\begin{proof}
(1) This is clear by definition.

(2) Assume indeed that $G,H$ are basic. For the product assertion, if $G,H$ are classical then $G\times H$ is classical, and we are done. If not, assume for instance that $\widehat{\Gamma}\subset G$ is non-classical. Then $\widehat{\Gamma}\times 1\subset G\times H$ is non-classical, and we are done again.

The generating assertion follows similarily, by replacing $\times$ by $<,>$.

(3) Assume that $G,H$ are basic. For the free product assertion, if $G=\{1\}$ or $H=\{1\}$ we are done. If not, assume first that both $G,H$ are classical. If we pick subgroups $\widehat{\mathbb Z}_a\subset G$ and $\widehat{\mathbb Z}_b\subset H$ with $a,b\in\{2,3,\ldots,\infty\}$, with the convention $\mathbb Z_\infty=\mathbb Z$, then $\Gamma=\mathbb Z_a*\mathbb Z_b$ is non-abelian and $\widehat{\Gamma}\subset G\;\hat{*}\;H$, and we are done again. Finally, if for instance $\widehat{\Gamma}\subset G$ is non-classical, then $\widehat{\Gamma}\subset G\;\hat{*}\;H$ is non-classical either, and we are done again.

The free wreath product assertion follows similarily, by replacing Wang's dual free product operation $\hat{*}$ with Bichon's free wreath product operation $\wr_*$ from \cite{bi1}. 
\end{proof}

In fact, most of known compact quantum groups are basic. Here is a verification for the main examples of ``easy'' quantum groups, introduced in \cite{bsp} and studied in \cite{bcs}:

\begin{proposition}
The main examples of easy quantum groups are all basic:
\begin{enumerate}
\item The classical ones: $O_n,S_n,H_n,B_n,S_n',B_n'$.

\item The free ones: $O_n^+,S_n^+,H_n^+,B_n^+,S_n'^+,B_n'^+$.

\item The half-liberated ones: $O_n^*,H_n^*$.
\end{enumerate}
\end{proposition}

\begin{proof}
We refer to the papers \cite{bcs}, \cite{bsp} for the definition of easiness, and for the precise construction and interpretation of the above quantum groups.

(1) Any classical group is basic by definition.

(2) The free examples, and the inclusions between them, are as follows:
$$\begin{matrix}
B_n^+&\subset&B_n'^+&\subset&O_n^+\\
\\
\cup&&\cup&&\cup\\
\\
S_n^+&\subset&S_n'^+&\subset&H_n^+
\end{matrix}$$

Let us first look at the case $n=2$. Here, according to \cite{bsp}, the diagram is:
$$\begin{matrix}
\mathbb Z_2&\subset&\widehat{D}_\infty&\subset&O_2^+\\
\\
\cup&&\cup&&\cup\\
\\
\mathbb Z_2&\subset&\mathbb Z_2\times\mathbb Z_2&\subset&O_2^{-1}
\end{matrix}$$

Thus all these quantum groups are basic, except perhaps for $H_2^+=O_2^{-1}$. But this quantum group is basic too, because it is known that the exceptional embedding $B_2'\subset H_2$ has a free analogue $B_2'^+\subset H_2^+$, which reads $\widehat{D}_\infty\subset O_2^{-1}$. At $n=3$ now, the diagram is:
$$\begin{matrix}
O_2^+&\subset&\mathbb Z_2\;\hat{*}\;O_2^+&\subset&O_3^+\\
\\
\cup&&\cup&&\cup\\
\\
S_3&\subset&\mathbb Z_2\times S_3&\subset&H_3^+
\end{matrix}$$

Here we have used the isomorphism $B_n^+\simeq O_{n-1}^+$, cf. \cite{rau}. Now since $O_2^+$ is basic, so are the other 2 quantum groups in the upper row. As for the remaining non-classical quantum group, namely $H_3^+$, this is basic too, because it contains $H_2^+$.

Finally, at $n\geq 4$ we can use the standard embeddings $\widehat{D}_\infty\subset S_4^+\subset S_n^+\subset G_n^+$ in order to conclude that any free quantum group $G_n^+$ contains $\widehat{D}_\infty$, and hence is basic.

(3) At $n=2$ it is known from \cite{bve} that we have $O_2^*=O_2^+$. Thus we have as well $H_2^*=H_2^+$, and since we already know that $O_2^+,H_2^+$ are basic, we are done.

At $n\geq 3$ now, consider the group $L_n=\mathbb Z_2^{*n}/<abc=cba>$, where the relations $abc=cba$ are imposed to the standard generators of $\mathbb Z_2^{*n}$. By \cite{bve} this group is not abelian, and we have inclusions $\widehat{L}_n\subset H_n^*\subset O_n^*$. Thus $O_n^*,H_n^*$ are basic, and we are done.
\end{proof}

It is possible to prove that some other quantum groups from \cite{bcs} are basic as well. So, the following question appears: are there any examples of strange quantum groups?

This is quite a tricky question, and the simplest answer comes from:

\begin{theorem}
The following quantum groups are strange:
\begin{enumerate}
\item The Kac-Paljutkin quantum group \cite{kpa}.

\item Its generalizations constructed by Sekine in \cite{sek}.
\end{enumerate}
\end{theorem}

\begin{proof}
The quantum groups in the statement are known to provide counterexamples to the ``quantum version'' of a classical result of Kawada and It\^o \cite{kit}, stating that all the idempotent states must come from subgroups. More precisely:

(1) Pal computed in \cite{pal} all the quantum subgroups of the Kac-Paljutkin quantum group, and constructed an idempotent state not coming from them. But his classification can be used as well for our purposes, because all the group dual subgroups found in \cite{pal} are abelian, and hence the Kac-Paljutkin quantum group is strange.

(2) The situation here is very similar, with the classification of all the quantum subgroups, which implies strangeness, done by Franz and Skalski in \cite{fsk}.
\end{proof}

Some other strange examples include the quantum double torus, introduced by Hajac and Masuda in \cite{hma}, for irrational values of the rotation parameter. We will come back a bit later to these examples, after developing some general group dual subgroup theory.

\section{Diagonal subgroups}

We have seen in the previous section that one question of interest is that of classifying the group dual subgroups of a given quantum group $G$. Indeed, once such a classification is available, the question of deciding whether $G$ is basic or not becomes trivial. In fact, this latter question doesn't seem to be much simpler than the classification one.

In this section we develop a number of techniques for dealing with this problem, in the ``matrix'' case. Let us first recall the following definition, due to Wang \cite{wa1}:

\begin{definition}
$C(U_n^+)$ is the universal $C^*$-algebra generated by variables $u_{ij}$ with $i,j=1,\ldots,n$ with the relations making $u=(u_{ij})$ and $u^t=(u_{ji})$ unitary matrices.
\end{definition}

As a first observation, this algebra is a Hopf $C^*$-algebra in the sense of \cite{wo1}, hence in the sense of \cite{wo3} as well, with comultiplication, counit and antipode given by:
\begin{eqnarray*}
\Delta(u_{ij})&=&\sum_ku_{ik}\otimes u_{kj}\\
\varepsilon(u_{ij})&=&\delta_{ij}\\ 
S(u_{ij})&=&u_{ji}^*
\end{eqnarray*}

Observe the similarity with the usual formulae for the matrix multiplication, unit and inversion. In fact, given any compact group of matrices $G\subset U_n$, we have a surjective morphism of Hopf $C^*$-algebras $\pi:C(U_n^+)\to C(G)$ given by $\pi(u_{ij}):g\to g_{ij}$.

We will need the following basic result, due to Woronowicz \cite{wo1}:

\begin{proposition}
Let $\Lambda=<g_1,\ldots,g_n>$ be a discrete group, and set $D=diag(g_i)$.
\begin{enumerate}
\item We have a morphism $\pi:C(U_n^+)\to C^*(\Lambda)$ given by $(id\otimes\pi)u=D$.

\item In fact, for any $U\in U_n$ we have such a morphism, given by $(id\otimes\pi)u=UDU^*$.

\item All the group dual subgroups $\widehat{\Lambda}\subset U_n^+$ appear from morphisms as in (2).
\end{enumerate}
\end{proposition}

\begin{proof}
(1) follows from (2), which follows from the fact that $V=UDU^*$ is unitary, with unitary transpose. As for (3), this follows from the representation theory results in \cite{wo1}. Indeed, an embedding $\widehat{\Lambda}\subset U_n^+$ must come from a surjective morphism $\pi:C(U_n^+)\to C^*(\Lambda)$. Now since the matrix $V=(id\otimes\pi)u$ is a unitary corepresentation of $C^*(\Lambda)$, we can find a unitary $U\in U_n$ such that $D=U^*VU$ is a direct sum of irreducible corepresentations. But these irreducible corepresentations are known to be all 1-dimensional, and corresponding to the elements of $\Lambda$, so we have $D=diag(g_i)$ for certain elements $g_i\in\Lambda$. Moreover, since $\pi$ is surjective we have $\Lambda=<g_1,\ldots,g_n>$, and we are done.
\end{proof}

We will need as well the following basic result, from \cite{bve}:

\begin{proposition}
Let $G\subset U_n^+$ be a closed subgroup. 
\begin{enumerate}
\item The ideal $I=<u_{ij}|i\neq j>$ is a Hopf ideal.

\item The quotient algebra $A=C(G)/I$ is cocommutative.

\item The generators $g_i=u_{ii}$ of the algebra $A$ are group-like.

\item We have $A=C^*(\Gamma_1)$, where $\Gamma_1=<g_1,\ldots,g_n>$.
\end{enumerate}
\end{proposition}

\begin{proof}
These assertions are more or less equivalent, and follow from the fact that, when dividing by $I$, the relation $\Delta(u_{ii})=\sum_ku_{ik}\otimes u_{ki}$ becomes $\Delta(u_{ii})=u_{ii}\otimes u_{ii}$. See \cite{bve}.
\end{proof}

We should mention that in the above result we identify as usual the full and reduced versions of our Hopf $C^*$-algebras, so that the equality in (4) means that the full version of $A$ equals the full group algebra $C^*(\Gamma_1)$. For more on this subject, see \cite{wa3}.

We can combine the above two results, in the following way:

\begin{definition}
Let $G=\widehat{\Gamma}$ be a closed subgroup of $U_n^+$. Associated to any unitary matrix $U\in U_n$ is the classical discrete group quotient $\Gamma\to\Gamma_U$ given by
$$C^*(\Gamma_U)=C(G)/<v_{ij}=0,\forall\,i\neq j>$$
where $v=UuU^*$. Also, we write $\Gamma_U=<g_1,\ldots,g_n>$, where $g_i=v_{ii}$.
\end{definition}

Observe the compatibility with Proposition 4.3. Indeed, the discrete group $\Gamma_1$ constructed there coincides with the discrete group $\Gamma_U$ constructed here, at $U=1$.

We can state now our main theoretical observation:

\begin{proposition}
Let $G\subset U_n^+$ be a closed subgroup.
\begin{enumerate}
\item The group dual subgroups $\widehat{\Lambda}\subset G$ come from the quotients of type $\Gamma_U\to\Lambda$.

\item $G$ is strange if and only if it is not classical, and all groups $\Gamma_U$ are abelian.
\end{enumerate}
\end{proposition}

\begin{proof}
(1) We have by definition an embedding $\widehat{\Gamma}_U\subset G$ for any $U\in U_n$, so if we take a quotient group $\Gamma_U\to\Lambda$ then we will have embeddings $\widehat{\Lambda}\subset \widehat{\Gamma}_U\subset G$. 

Conversely, assume that we have a group dual subgroup $\widehat{\Lambda}\subset G$. Thus we have embeddings $\widehat{\Lambda}\subset G\subset U_n^+$, and Proposition 4.2 tells us, the corresponding surjection $\varphi:C(U_n^+)\to C^*(\Lambda)$ must be of the form $(id\otimes\varphi)u=UDU^*$, where $D=diag(h_1,\ldots,h_n)$ is a diagonal matrix formed by a family of generators of $\Lambda$, and $U\in U_n$. With this choice of $U\in U_n$, we have a surjective map $\Gamma_U\to\Lambda$ given by $g_i\to h_i$, and we are done.

(2) This follows from (1), because if a group is abelian, then so are all its quotients.
\end{proof}

As a first application, consider the quantum double torus algebra $Q=C(\mathbb T^2)\oplus A_{2 \theta}$, constructed by Hajac and Masuda in \cite{hma}. If we denote by $A,D$ the standard generators of $C(\mathbb T^2)$ and by $B,C$ the standard generators of $ A_{2 \theta}$, then the comultiplication of $Q$ is by definition the one making $V=(^A_B{\ }^C_D)$ a corepresentation. See \cite{hma}.

\begin{theorem}
The quantum double torus is strange for $\theta\notin 2\pi\mathbb Z$.
\end{theorem}

\begin{proof}
This follows by computing the diagonal subgroups, and by using Proposition 4.5. Consider indeed an arbitrary matrix $U\in U_2$. With $d=\det U$, we can write:
$$U=d\begin{pmatrix}s&t\\ -\bar{t}&\bar{s}\end{pmatrix}$$

Here $s,t$ are certain complex numbers satisfying $|s|^2+|t|^2=1$. We have:
$$UVU^*=\begin{pmatrix}
s\bar{s}A+\bar{s}tB+s\bar{t}C+t\bar{t}D&-stA-t^2B+s^2C+stD\\
-\bar{s}\bar{t}A+\bar{s}^2B-\bar{t}^2C+\bar{s}\bar{t}D&t\bar{t}A-s\bar{t}B-s\bar{t}C+s\bar{s}D
\end{pmatrix}$$

We know that $C^*(\Gamma_U)$ is the quotient of $Q$ by the relations making vanish the off-diagonal entries of $UVU^*$. Now, according to the above formula, these relations are:
$$st(A-D)=-t^2B+s^2C$$
$$\bar{s}\bar{t}(A-D)=\bar{s}^2B-\bar{t}^2C$$

By multiplying the first relation by $\bar{s}\bar{t}$ and the second one by $st$ we obtain:
$$\bar{s}\bar{t}(-t^2B+s^2C)=st(\bar{s}^2B-\bar{t}^2C)$$

Thus we have $s\bar{t}(s\bar{s}+t\bar{t})C=\bar{s}t(s\bar{s}+t\bar{t})B$, and by dividing by $s\bar{s}+t\bar{t}=1$ we obtain $s\bar{t}C=\bar{s}tB$. Thus, in the case $s,t\neq 0$, the elements $B,C$ are proportional.

On the other hand, we know that $BC=e^{i\theta}CB$. Thus in the case $s,t\neq 0$ we obtain $B=C=0$, so the quotient is generated by $A,D$, which commute, and we are done.

Finally, in the case $s=0$ or $t=0$ the above two relations defining $C^*(\Gamma_U)$ simply become $B=C=0$, so the same argument applies, and we are done.
\end{proof}

\section{Explicit computations}

In this section we present some explicit computations of the family of discrete groups $F=\{\Gamma_U|U\in U_n\}$ associated to a compact quantum group $G\subset U_n^+$. Our main result here will concern the case where $G=S_n^+$ is Wang's quantum permutation group \cite{wa1}.

We have first the following basic result, in the group dual case:

\begin{proposition}
Let $\Gamma=<g_1,\ldots,g_n>$ be a discrete group, and regard $G=\widehat{\Gamma}$ as a closed subgroup of $U_n^+$, by using the biunitary matrix $D=diag(g_i)$. Then:
$$\Gamma_U=\Gamma/<g_s=g_t|\exists j,\, U_{jt}\neq 0,U_{js}\neq 0>$$
\end{proposition}

\begin{proof}
We know that $C^*(\Gamma_U)$ is the quotient of $C^*(\Gamma)$ by the relations making vanish the off-diagonal entries of the matrix $UDU^*$. But this matrix is:
$$(UDU^*)_{ij}=\sum_kU_{ik}\bar{U}_{jk}g_k$$

Let now $t\in\{1,\ldots,n\}$. By multiplying by $\bar{U}_{it}$ and summing over $i$ we get:
\begin{eqnarray*}
\sum_i\bar{U}_{it}(UDU^*)_{ij}
&=&\sum_i\sum_k\bar{U}_{it}U_{ik}\bar{U}_{jk}g_k\\
&=&\sum_k\bar{U}_{jk}g_k\sum_i\bar{U}_{it}U_{ik}\\
&=&\bar{U}_{jt}g_t
\end{eqnarray*}

Now assume that we are in the quotient algebra $C^*(\Gamma_U)$. Since the off-diagonal entries of $UDU^*$ vanish here, the above formula becomes $\bar{U}_{jt}(UDU^*)_{jj}=\bar{U}_{jt}g_t$, so we get:
$$\bar{U}_{jt}\sum_k|U_{jk}|^2g_k=\bar{U}_{jt}g_t$$

In particular, for any $j,t$ such that $\bar{U}_{jt}\neq 0$, we must have:
$$g_t=\sum_k|U_{jk}|^2g_k$$

Now fix an index $j\in\{1,\ldots,n\}$. Since the expression on the right is independent on $t$, we conclude that the elements $g_t$, with $t\in\{1,\ldots,n\}$ having the property that $U_{jt}\neq 0$, are all equal. So, in other words, $\Gamma_U$ is a quotient of the group in the statement.

In order to finish now, consider the group in the statement. We must prove that the off-diagonal coefficients of $UDU^*$ vanish. So, let us look at these coefficients:
$$(UDU^*)_{ij}=\sum_kU_{ik}\bar{U}_{jk}g_k$$

In this sum $k$ ranges over the set $S=\{1,\ldots,n\}$, but we can of course restrict the attention to the subset $S'$ of indices having the property $U_{ik}U_{jk}\neq 0$.  But for these latter indices the elements $g_k$ are all equal, say to an element $g\in\Gamma_U$, and we obtain:
$$(UDU^*)_{ij}=\left(\sum_{k\in S'}U_{ik}\bar{U}_{jk}\right)g=\left(\sum_{k\in S}U_{ik}\bar{U}_{jk}\right)g=\delta_{ij}g_i$$

This finishes the proof.
\end{proof}

Observe the similarity between the above statement and proof and the group dual computations in \cite{bss}. In fact, all these considerations seem to belong to a wider circle of ideas, including as well the computation of the groups $\Gamma_U$ for the quantum permutation group $G=S_n^+$. Indeed, let us first recall the following key result, due to Bichon \cite{bi2}:

\begin{proposition}
Let $\mathbb Z_{n_1}*\ldots*\mathbb Z_{n_k}\to\Lambda$ be a quotient group. Then $\widehat{\Lambda}\curvearrowright\mathbb C^n$, where $n=n_1+\ldots+n_k$, and any group dual coaction on $\mathbb C^n$ appears in this way.
\end{proposition}

\begin{proof}
First, by taking the dual free product of the canonical coactions $\mathbb Z_{n_i}\curvearrowright\mathbb C^{n_i}$, coming from the usual group embeddings $\mathbb Z_{n_i}\subset S_{n_i}$, we obtain a coaction as follows:
$$\mathbb Z_{n_1}\;\hat{*}\;\ldots\;\hat{*}\;\mathbb Z_{n_k}\curvearrowright\mathbb C^{n_1}\oplus\ldots\oplus\mathbb C^{n_k}$$

Thus with $\Gamma=\mathbb Z_{n_1}*\ldots*\mathbb Z_{n_k}$ we have a coaction $\widehat{\Gamma}\curvearrowright\mathbb C^n$, and it follows that for any quotient group $\Gamma\to\Lambda$ we have a coaction $\widehat{\Lambda}\curvearrowright \mathbb C^n$ as in the statement.

Conversely, assume $\widehat{\Lambda}\curvearrowright\mathbb C^n$. The fixed point algebra of this coaction must be of the form $A=\mathbb C^{n_1}\oplus\ldots\oplus\mathbb C^{n_k}$, with $n=n_1+\ldots+n_k$. Now since any faithful ergodic coaction $\widehat{\Lambda}\curvearrowright\mathbb C^r$ must come from a quotient group $\mathbb Z_r\to\Lambda$, this gives the result. See \cite{bi2}.
\end{proof}

In terms of diagonal subgroups now, we have the following result:

\begin{theorem} 
For a quantum permutation group $G=S_n^+$, the discrete group $\Gamma_U$ is generated by elements $g_1,\ldots,g_n$ with the relations
$$\begin{cases}
g_i=1&{\rm if}\ c_i\neq 0\\
g_ig_j=1&{\rm if}\ c_{ij}\neq 0\\ 
g_ig_j=g_k&{\rm if}\ d_{ijk}\neq 0
\end{cases}$$
where $c_i=\sum_lU_{il}$, $c_{ij}=\sum_lU_{il}U_{jl}$, $d_{ijk}=\sum_l\bar{U}_{il}\bar{U}_{jl}U_{kl}$.
\end{theorem}

\begin{proof} 
Fix $U\in U_n$, and write $w=UvU^*$, where $v$ is the fundamental representation of $S_n^+$. Let $X$ be an $n$-element set, and $\alpha$ be the coaction of $C(S_n^+)$ on $C(X)$. Write:
$$\varphi_i=\sum_l\bar{U}_{il}\delta_l\in C(X)$$

Also, let $g_i=(UvU^*)_{ii}\in C^*(\Gamma_U)$. If $\beta$ is the restriction of $\alpha$ to $C^*(\Gamma_U)$, then:
$$\beta(\varphi_i)=\varphi_i\otimes g_i$$

Now $C(X)$ is the universal $C^*$-algebra generated by elements $\delta_1,\ldots,\delta_n$ which are mutually orthogonal self-adjoint projections. Writing these conditions in terms of the linearly independent elements $\varphi_i$ by means of the formulae $\delta_i=\sum_lU_{il}\varphi_l$, we find that the universal relations for $C(X)$ in terms of the elements $\varphi_i$ are as follows:
$$\sum_ic_i\varphi_i=1,\quad \varphi_i^*=\sum_jc_{ij}\varphi_j,\quad \varphi_i\varphi_j=\sum_kd_{ijk}\varphi_k$$

Let $\tilde{\Gamma}_U$ be the group in the statement. Since $\beta$ preserves these relations, we get:
$$c_i(g_i-1)=0,\quad c_{ij}(g_ig_j-1)=0,\quad d_{ijk}(g_ig_j-g_k)=0$$

Thus $\Gamma_U$ is a quotient of $\tilde{\Gamma}_U$. On the other hand, it is immediate that we have a coaction $C(X)\to C(X)\otimes C^*(\tilde{\Gamma}_U)$, hence $C(\tilde{\Gamma}_U)$ is a quotient of $C(S_n^+)$. Since $w$ is the fundamental corepresentation of $S_n^+$ with respect to the basis $\{\varphi_i\}$, it follows that the generator $w_{ii}$ is sent to $\tilde{g}_i\in\tilde{\Gamma}_U$, while $w_{ij}$ is sent to zero. We conclude that $\tilde{\Gamma}_U$ is a quotient of $\Gamma_U$. Since the above quotient maps send generators on generators, we conclude that $\Gamma_U=\tilde{\Gamma}_U$.
\end{proof}

As an example, let us work out the case where $U$ is a direct sum of Fourier matrices. We obtain here the class of maximal group dual subgroups of $S_n^+$:

\begin{proposition}
Let $U=diag(F_{n_1},\ldots,F_{n_k})$, where $F_r=(\xi^{ij})/\sqrt{r}$ with $\xi=e^{2\pi i/r}$ is the Fourier matrix. Then for $G=S_n^+$ with $n=\sum n_i$ we have $\Gamma_U=\mathbb Z_{n_1}*\ldots*\mathbb Z_{n_k}$.
\end{proposition}

\begin{proof}
We apply Theorem 5.3, with the index set $X$ chosen to be $X=\mathbb Z_{n_1}\sqcup\ldots\sqcup\mathbb Z_{n_k}$. First, we have $c_i=\delta_{i0}$ for any $i$. Also, $c_{ij}=0$ unless $i,j,k$ belong to the same block to $U$, in which case $c_{ij}=\delta_{i+j,0}$, and $d_{ijk} =0$ unless $i,j,k$ belong to the same block of $U$, in which case $d_{ijk}=\delta_{i+j,k}$. Thus $\Gamma_U$ is the free product of $k$ groups which have generating relations $g_ig_j=g_{i+j}$ and $g_i^{-1}=g_{-i}$, so that $\Gamma_U=\mathbb Z_{n_1}*\ldots*\mathbb Z_{n_k}$, as stated.
\end{proof}

Finally, let us mention that the method in the proof of Theorem 5.3 applies as well to the more general situation where $G=G^+(X)$ is the quantum automorphism group of a finite noncommutative set $X$. This will be discussed in detail somewhere else.

\section{Concluding remarks}

We have seen in this paper that, given a compact connected Riemannian manifold $M$, its quantum isometry group $G=G^+(M)$ cannot contain non-classical group duals $\widehat{\Lambda}$. 

In addition, we have seen that in the case $G\subset U_n^+$, the classification of group dual subgroups of $G$ reduces to the computation of a certain family $F=\{\Gamma_U|U\in U_n\}$ of groups associated to $G$, which appear as ``universal objects'' for the problem $\widehat{\Lambda}\subset G$.

The computation of this family $F$ is therefore a key problem, that we solved here in the group dual case $G=\widehat{\Gamma}$, and in the quantum permutation group case $G=S_n^+$. 

The main question that we would like to address here concerns of course the potential unification of these computations. But we do not have further results here.

We have as well some other questions, of more theoretical nature. For instance the family of maximal group dual subgroups of a given compact quantum group $G$, taken modulo isomorphism, seems to be always finite. We do not know if this is the case.

\end{document}